\documentclass[reqno,12pt]{amsart}

\usepackage{amsmath}
\usepackage{amsfonts}
\usepackage{amssymb,euscript, mathrsfs,amscd} 
\usepackage{epsfig}
\usepackage{a4wide}
\usepackage{relsize}
\usepackage{blkarray}
\usepackage{multirow}
\usepackage{mathtools}
\usepackage{bm}
\usepackage{ytableau}
\usepackage{caption}
\usepackage{hyperref}
\usepackage{cases}
\usepackage{enumitem}

\usepackage{tikz}
\usepackage{tikz-cd}

\usepackage{subfigure}
\usepackage{graphicx} %ifpdf

\usetikzlibrary{decorations.markings}

\newtheorem{thm}{Theorem}[section]
\newtheorem{prop}[thm]{Proposition}
\newtheorem{lem}[thm]{Lemma}
\newtheorem{cor}[thm]{Corollary}

\theoremstyle{definition}
\newtheorem{defn}[thm]{Definition}
\newtheorem{ex}[thm]{Example}

\newcommand\ZZ{\mathbb{Z}}
\newcommand\AO{\mathcal{A}}
\newcommand\BB{\mathcal{B}}
\newcommand\CC{\mathcal{C}}
\newcommand\RR{\mathcal{R}}
\newcommand\oo{\mathfrak{o}}
\newcommand\VV{\mathbf{V}}
\newcommand\xx{\mathbf{x}}
\newcommand\bgamma{{\boldsymbol{\gamma}}}
\newcommand\qand{\quad\mbox{and}\quad}

\title[Acyclic orientation polynomials]{Acyclic orientation polynomials and the sink theorem for chromatic symmetric functions}

\author{Byung-Hak Hwang}
\address{Department of Mathematical Sciences, Seoul National University, Seoul,
South Korea}
\email{xoda@snu.ac.kr}

\author{Woo-Seok Jung}
\address{
Department of Mathematics, Sogang University, Seoul,
South Korea}
\email{jungws@sogang.ac.kr}

\author{Kang-Ju Lee}
\address{Department of Mathematical Sciences and Research Institute of Mathematics, Seoul
National University, Seoul, South Korea}
\email{leekj0706@snu.ac.kr}

\author{Jaeseong Oh}
\address{Department of Mathematical Sciences, Seoul National University, Seoul,
South Korea}
\email{jaeseong\_oh@snu.ac.kr}

\author{Sang-Hoon Yu}
\address{Department of Mathematical Sciences, Seoul National University, Seoul,
South Korea}
\email{ysh4017@snu.ac.kr}


\begin{document}
\begin{abstract}
We define the acyclic orientation polynomial of a graph to be the generating function for the sinks of its acyclic orientations.
Stanley proved that the number of acyclic orientations is equal to the chromatic polynomial evaluated at $-1$ up to sign.
Motivated by this link between acyclic orientations and the chromatic polynomial, we develop ``acyclic orientation'' analogues of theorems concerning the chromatic polynomial of Birkhoff, Whitney, and Greene-Zaslavsky.
As an application, we provide a new proof for Stanley's sink theorem for chromatic symmetric functions $X_G$. This theorem gives a relation between the number of acyclic orientations with a fixed number of sinks and the coefficients in the expansion of $X_G$ with respect to elementary symmetric functions.
\end{abstract}

\maketitle

\section{Introduction}
The purpose of this paper is to introduce acyclic orientation polynomials, present several expressions for them, and give a new proof for Stanley's sink theorem using these expressions. Throughout this paper, let $G=(V,E)$ be a graph with $|V|=d$ vertices. We allow $G$ to contain multiple edges or loops. 

Our object of study is an \emph{acyclic orientation} of the graph $G$, an assignment of a direction to each edge such that the orientation induces no directed cycles. Denote by $\mathcal{A}(G)$ the collection of acyclic orientations of $G$.
The number of acyclic orientations of $G$ is a Tutte-Grothendieck invariant, i.e., this number obeys a \emph{deletion-contraction} recursion.
As a refinement of this quantity, we introduce the \emph{acyclic orientation polynomial} of $G$.
For $\mathfrak{o} \in \mathcal{A}(G)$, a vertex $v$ is called a \emph{sink} if the direction of each edge incident with $v$ is toward $v$.
Let $\operatorname{Sink}(\mathfrak{o})$ be the set of sinks of $\mathfrak{o}$.
We associate a formal variable to each vertex $v\in V$, and denote this variable by $v$ using the same notation for the vertex. Also we assume that these variables commute with each other, so $\ZZ[V]$ denotes the ring of polynomials in $V$ with integer coefficients.
The acyclic orientation polynomial of $G$ (Definition~\ref{def:AOP}) is defined to be
\[
A_G(V)=\sum_{\mathfrak{o}\in \mathcal{A}(G)}\prod_{v \in \operatorname{Sink}(\mathfrak{o})}{v} \in \ZZ[V].
\]
Specializing $v=t$ for all $v \in V$ in $A_G(V)$, we obtain the polynomial $a_G(t)$ whose coefficient of $t^j$ counts the number of acyclic orientations with $j$ sinks \cite{Sta95, Lass01}.

One might ask if there is a deletion-contraction recurrence for the acyclic orientation polynomial, as with the number of acyclic orientations. To answer this, we introduce the following assignment:
\begin{equation}\label{eq:intro_v_e}
v_e=u_1+u_2-u_1u_2, \mbox{ or } (1-v_e)=(1-u_1)(1-u_2), 
\end{equation}
as an element of $\ZZ[V]$. Here, $v_e$ denotes the vertex of the graph $G/e$ obtained from $G$ by contracting an edge $e=u_1u_2 \in E$. Using this assignment, we can regard the acyclic orientation polynomial of $G/e$ as a polynomial in $V$, and then we establish a deletion-contraction recurrence for the acyclic orientation polynomial (Theorem~\ref{thm:DC}).
This assignment leads us to the notion of a generalized deletion-contraction recurrence.
This generalized recurrence defines the \emph{$\VV$-polynomial} in a more general setting \cite{EMM11}, and then our polynomial $A_G(V)$ can be obtained by specializing the $\VV$-polynomial. See Appendix~\ref{sec:appendix} for details.

In \cite{Sta73}, Stanley showed that the number of acyclic orientations of $G$ is equal to $(-1)^d{\chi_G(-1)}$, where $\chi_G(n)$ is the chromatic polynomial of $G$. This result motivates us to develop ``acyclic orientation'' analogues of theorems concerning $\chi_G(n)$. Let us recall four famous expressions for the chromatic polynomial $\chi_G(n)$:
\begin{align*}
\chi_G(n)
&= \sum_{S \subseteq E}(-1)^{|S|} n^{c(S)} && \mbox{[subgraph expansion]} \\ 
&= \sum_{S \in \BB_G}(-1)^{|S|} n^{c(S)} && \mbox{{\cite[Whitney's Theorem]{Whi32}}} \\ 
&= \sum_{\pi \in L_G}\mu_G(\hat{0},\pi) n^{|\pi|} && \mbox{{\cite[Birkhoff's Theorem]{Bir12}}} \\ 
&= \sum_{\mathfrak{o} \in \mathcal{A}(G)}(-1)^{d-|\pi(\mathfrak{o})|} n^{|\pi(\mathfrak{o})|} && \mbox{{\cite[Corollary 7.4]{GZ83}}},
\end{align*}
where $S\subseteq E$ is a spanning subgraph of $G$, $|S|$ is the number of edges of $S$, $c(S)$ is the number of connected components of $S$, $\BB_G$ is the broken circuit complex of $G$, $L_G$ is the bond lattice of $G$ and $\mu_G$ its M\"{o}bius function, and $\pi:\mathcal{A}(G) \to L_G$ is the sink-component map defined in \cite[Section 4]{BN20}. These notions will be defined in Section~\ref{sec:Expression}.

In this paper, we provide four similar expressions for $A_G(V)$:
\begin{align*}
A_G(V) &= \sum_{S \subseteq E}(-1)^{s(S)}\prod_{C \in \mathcal{C}(S)}v_C && \mbox{[Theorem~\ref{thm:AC_Sub}]} \\ 
&= \sum_{S \in \BB_G}\prod_{C \in \mathcal{C}(S)}v_C && \mbox{[Theorem~\ref{thm:AC_BC}]} \\
&= \sum_{\pi \in L_G}(-1)^{d-|\pi|}\mu_G{(\hat{0},\pi)}\prod_{B \in \pi}{v_B} && \mbox{[Theorem~\ref{thm:AC_LG}]}\\ 
&= \sum_{\mathfrak{o} \in \mathcal{A}(G)}\prod_{ B \in \pi(\mathfrak{o})}v_{B} && \mbox{[Theorem~\ref{thm:AC_SD}]},
\end{align*}
where $\mathcal{C}(S)$ is the set of connected components of a subgraph $S$ and $s(S)=|S|-d+c(S)$ is the corank of $S$. To a connected component $C$ and a vertex subset $B$, we define
\[
v_C=1-\prod_{v \in V(C)}(1-v) \in\ZZ[V] \quad\mbox{and}\quad v_B=1-\prod_{v \in B}(1-v) \in\ZZ[V],
\]
which are generalizations of \eqref{eq:intro_v_e}. Comparing the above expressions for $\chi_G(n)$ and $A_G(V)$, we can see that replacing $n$ by $-v_C$ or $-v_B$ in expressions for the chromatic polynomial gives the acyclic orientation polynomial (up to sign).

Our main application of various expressions for $A_G(V)$ is to give a \emph{new} proof of Stanley's sink theorem \cite[Theorem 3.3]{Sta95} for the chromatic symmetric function $X_G$. This theorem states that
\begin{equation*}
\operatorname{sink}(G,j)=\sum_{\substack{\lambda\vdash d \\ \ell(\lambda)=j}}c_\lambda,
\end{equation*}
where $\operatorname{sink}(G,j)$ is the number of acyclic orientations of $G$ with $j$
sinks, the numbers $c_{\lambda}$ are defined by the expansion $X_G=\sum_{\lambda\vdash d}c_{\lambda}e_{\lambda}$ in terms of elementary symmetric functions $e_{\lambda}$, and $\ell(\lambda)$ is the length of a partition $\lambda$.
Stanley's original proof relies on the theory of quasi-symmetric functions and $P$-partitions, which led him \cite{Sta99, Sta12} to ask for a simple conceptual proof of the theorem.

This paper is organized as follows. Section~\ref{sec:AC_DC} introduces the acyclic orientation polynomial $A_G(V)$ and proves that it satisfies a deletion-contraction recurrence. Section \ref{sec:Expression} presents various expressions for $A_G(V)$. Section~\ref{sec:Stanley} provides a new proof of Stanley's sink theorem for chromatic symmetric functions. Section~\ref{sec:invariant} considers to what extent $A_G(V)$ and $a_G(t)$ distinguish non-isomorphic graphs. In Appendix~\ref{sec:appendix}, we discuss a relation between the acyclic orientation polynomial and the $\VV$-polynomial.

\section{Acyclic orientation polynomials and their recurrence}\label{sec:AC_DC}
For an edge $u_1u_2 \in E$, a direction $\overrightarrow{u_1u_2}$ (resp. $\overrightarrow{u_2u_1}$) means that the edge is oriented toward $u_2$ (resp. $u_1$). An \emph{orientation} $\mathfrak{o}$ is an assignment of a direction $\overrightarrow{u_1u_2}$ or $\overrightarrow{u_2u_1}$ to each edge $u_1u_2 \in E$. An orientation $\mathfrak{o}$ is said to be \emph{acyclic} if $\mathfrak{o}$ has no directed cycles. Let $\mathcal{A}(G)$ be the set of acyclic orientations of $G$. For $\mathfrak{o} \in \mathcal{A}(G)$, a \emph{sink} of $\mathfrak{o}$ is a vertex $v$ such that the direction of each edge incident with $v$ is toward $v$. Let $\operatorname{Sink}(\mathfrak{o})$ be the set of sinks of an orientation $\mathfrak{o}$ and $\operatorname{sink}(\mathfrak{o})=|\operatorname{Sink}(\mathfrak{o})|$. 

We associate a variable to each vertex $v \in V$, and use the same notation $v$ for this variable. Then $V$ also denotes the set of the variables corresponding to vertices. Assume that all the variables commute with each other.

We now introduce the main object in this paper.
\begin{defn}\label{def:AOP} \label{def:A_G}
 For a graph $G=(V,E)$, the \emph{acyclic orientation polynomial} $A_G(V)$ of $G$ is the generating function for sinks of acyclic orientations of $G$, i.e.,
\[
A_G(V)=\sum_{\mathfrak{o}\in \mathcal{A}(G)}\prod_{v \in \operatorname{Sink}(\mathfrak{o})}{v}.
\]
Also $a_G(t)$ is the polynomial obtained from $A_G(V)$ by setting $v=t$ for each $v \in V$, i.e.,
\[
a_G(t)=\sum_{\mathfrak{o}\in \mathcal{A}(G)}t^{\operatorname{sink}(\mathfrak{o})}.
\]
\end{defn}
Note that if $G$ contains a loop, then there is no acyclic orientation of $G$ so that $A_G(V) = 0$.

For a subset $U$ of $V$, let $\mathcal{A}(G,U)$ be the set of acyclic orientations $\mathfrak{o}$ of $G$ with $\operatorname{Sink}(\mathfrak{o})=U$. Then the coefficient of $\prod_{v \in U}v$ in $A_G(V)$ is equal to $|\mathcal{A}(G,U)|$, which will be denoted by $a(G,U)$. When $U$ consists of a single vertex $u$, we write $a(G,u)$ instead of $a(G,\{u\})$.

\begin{ex}\label{ex:two} Let us consider two graphs $G_1=(\{v_1,v_2,v_3\},\{v_1v_2,v_2v_3,v_1v_3\})$ and $G_2=(\{v_1,v_2,v_3,v_4\},\{v_1v_2,v_2v_3,v_1v_3,v_1v_4\})$. Their acyclic orientation polynomials are
\begin{align*}
A_{G_1}(\{v_1,v_2,v_3\}) &= 2(v_1+v_2+v_3), \mbox{ and} \\
A_{G_2}(\{v_1,v_2,v_3,v_4\}) &= 2(v_1+v_2+v_3+v_4+v_2v_4+v_3v_4).
\end{align*}
In Figure \ref{fig:AO_G_2}, we list all the acyclic orientations of $G_2$ with corresponding monomials.
\end{ex}

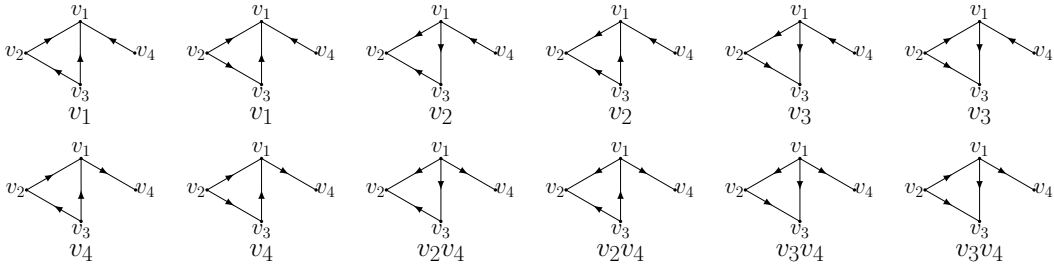
\begin{figure}[h]
\centering
\resizebox{.9\linewidth}{!}{
\begin{tabular}{cccccc}
%-------01_____v4-----------------
\begin{tikzpicture}[scale=0.50,every node/.style={transform shape}]
\fill (0,1) circle (1.5pt);
\fill (0,-1) circle (1.5pt);
\fill ({sqrt(3)},0) circle (1.5pt);
\fill (-{sqrt(3)},0) circle (1.5pt);

\node at (0,1.3) {\LARGE{$v_1$}};
\node at (-{sqrt(3)}-0.355,0) {\LARGE{$v_2$}};
\node at (0,-1.3) {\LARGE{$v_3$}};
\node at ({sqrt(3)+0.35},0) {\LARGE{$v_4$}};
%------------------------------
\begin{scope}[decoration={
    markings,
    mark=at position 0.50 with {\arrow{latex}}}
    ] 
\draw[postaction={decorate}] (-{sqrt(3)},0) -- (0,1);
\draw[postaction={decorate}] (0,-1) -- (-{sqrt(3)},0);
\draw[postaction={decorate}] (0,-1) -- (0,1);
\draw[postaction={decorate}] ({sqrt(3)},0) -- (0,1);
\end{scope}

\node at (0,-2) {\Huge{$v_1$}};
\end{tikzpicture}&
%-------02_____v3-----------------
\begin{tikzpicture}[scale=0.50,every node/.style={transform shape}]
\fill (0,1) circle (1.5pt);
\fill (0,-1) circle (1.5pt);
\fill ({sqrt(3)},0) circle (1.5pt);
\fill (-{sqrt(3)},0) circle (1.5pt);

\node at (0,1.3) {\LARGE{$v_1$}};
\node at (-{sqrt(3)}-0.35,0) {\LARGE{$v_2$}};
\node at (0,-1.3) {\LARGE{$v_3$}};
\node at ({sqrt(3)+0.3},0) {\LARGE{$v_4$}};
%------------------------------
\begin{scope}[decoration={
    markings,
    mark=at position 0.50 with {\arrow{latex}}}
    ] 
\draw[postaction={decorate}] (-{sqrt(3)},0) -- (0,1);
\draw[postaction={decorate}] (-{sqrt(3)},0) -- (0,-1);
\draw[postaction={decorate}] (0,-1) -- (0,1);
\draw[postaction={decorate}] ({sqrt(3)},0) -- (0,1);
\end{scope}

\node at (0,-2) {\Huge{$v_1$}};
\end{tikzpicture}&
%-------03_____v2v4-----------------
\begin{tikzpicture}[scale=0.50,every node/.style={transform shape}]
\fill (0,1) circle (1.5pt);
\fill (0,-1) circle (1.5pt);
\fill ({sqrt(3)},0) circle (1.5pt);
\fill (-{sqrt(3)},0) circle (1.5pt);

\node at (0,1.3) {\LARGE{$v_1$}};
\node at (-{sqrt(3)}-0.35,0) {\LARGE{$v_2$}};
\node at (0,-1.3) {\LARGE{$v_3$}};
\node at ({sqrt(3)+0.3},0) {\LARGE{$v_4$}};
%------------------------------
\begin{scope}[decoration={
    markings,
    mark=at position 0.50 with {\arrow{latex}}}
    ] 
\draw[postaction={decorate}] (0,1) -- (-{sqrt(3)},0);
\draw[postaction={decorate}] (0,-1) -- (-{sqrt(3)},0);
\draw[postaction={decorate}] (0,1) -- (0,-1);
\draw[postaction={decorate}] ({sqrt(3)},0) -- (0,1);
\end{scope}

\node at (0,-2) {\Huge{$v_2$}};
\end{tikzpicture}&
%-------04_____v2-----------------
\begin{tikzpicture}[scale=0.50,every node/.style={transform shape}]
\fill (0,1) circle (1.5pt);
\fill (0,-1) circle (1.5pt);
\fill ({sqrt(3)},0) circle (1.5pt);
\fill (-{sqrt(3)},0) circle (1.5pt);

\node at (0,1.3) {\LARGE{$v_1$}};
\node at (-{sqrt(3)}-0.35,0) {\LARGE{$v_2$}};
\node at (0,-1.3) {\LARGE{$v_3$}};
\node at ({sqrt(3)+0.3},0) {\LARGE{$v_4$}};
%------------------------------
\begin{scope}[decoration={
    markings,
    mark=at position 0.50 with {\arrow{latex}}}
    ] 
\draw[postaction={decorate}] (0,1) -- (-{sqrt(3)},0);
\draw[postaction={decorate}] (0,-1) -- (-{sqrt(3)},0);
\draw[postaction={decorate}] (0,-1) -- (0,1);
\draw[postaction={decorate}] ({sqrt(3)},0) -- (0,1);
\end{scope}

\node at (0,-2) {\Huge{$v_2$}};
\end{tikzpicture}&
%-------05_____v2v4-----------------
\begin{tikzpicture}[scale=0.50,every node/.style={transform shape}]
\fill (0,1) circle (1.5pt);
\fill (0,-1) circle (1.5pt);
\fill ({sqrt(3)},0) circle (1.5pt);
\fill (-{sqrt(3)},0) circle (1.5pt);

\node at (0,1.3) {\LARGE{$v_1$}};
\node at (-{sqrt(3)}-0.35,0) {\LARGE{$v_2$}};
\node at (0,-1.3) {\LARGE{$v_3$}};
\node at ({sqrt(3)+0.3},0) {\LARGE{$v_4$}};
%------------------------------
\begin{scope}[decoration={
    markings,
    mark=at position 0.50 with {\arrow{latex}}}
    ] 
\draw[postaction={decorate}] (0,1) -- (-{sqrt(3)},0);
\draw[postaction={decorate}] (-{sqrt(3)},0) -- (0,-1);
\draw[postaction={decorate}] (0,1) -- (0,-1);
\draw[postaction={decorate}] ({sqrt(3)},0) -- (0,1);
\end{scope}

\node at (0,-2) {\Huge{$v_3$}};
\end{tikzpicture}&
%-------06_____v2-----------------
\begin{tikzpicture}[scale=0.50,every node/.style={transform shape}]
\fill (0,1) circle (1.5pt);
\fill (0,-1) circle (1.5pt);
\fill ({sqrt(3)},0) circle (1.5pt);
\fill (-{sqrt(3)},0) circle (1.5pt);

\node at (0,1.3) {\LARGE{$v_1$}};
\node at (-{sqrt(3)}-0.35,0) {\LARGE{$v_2$}};
\node at (0,-1.3) {\LARGE{$v_3$}};
\node at ({sqrt(3)+0.3},0) {\LARGE{$v_4$}};
%------------------------------
\begin{scope}[decoration={
    markings,
    mark=at position 0.50 with {\arrow{latex}}}
    ] 
\draw[postaction={decorate}] (-{sqrt(3)},0) -- (0,1);
\draw[postaction={decorate}] (-{sqrt(3)},0) -- (0,-1);
\draw[postaction={decorate}] (0,1) -- (0,-1);
\draw[postaction={decorate}] ({sqrt(3)},0) -- (0,1);
\end{scope}

\node at (0,-2) {\Huge{$v_3$}};
\end{tikzpicture}\\
%-------07_____v4-----------------
\begin{tikzpicture}[scale=0.50,every node/.style={transform shape}]
\fill (0,1) circle (1.5pt);
\fill (0,-1) circle (1.5pt);
\fill ({sqrt(3)},0) circle (1.5pt);
\fill (-{sqrt(3)},0) circle (1.5pt);

\node at (0,1.3) {\LARGE{$v_1$}};
\node at (-{sqrt(3)}-0.35,0) {\LARGE{$v_2$}};
\node at (0,-1.3) {\LARGE{$v_3$}};
\node at ({sqrt(3)+0.3},0) {\LARGE{$v_4$}};
%------------------------------
\begin{scope}[decoration={
    markings,
    mark=at position 0.50 with {\arrow{latex}}}
    ] 
\draw[postaction={decorate}] (-{sqrt(3)},0) -- (0,1);
\draw[postaction={decorate}] (0,-1) -- (-{sqrt(3)},0);
\draw[postaction={decorate}] (0,-1) -- (0,1);
\draw[postaction={decorate}] (0,1) -- ({sqrt(3)},0);
\end{scope}

\node at (0,-2) {\Huge{$v_4$}};
\end{tikzpicture}&
%-------08_____v3-----------------
\begin{tikzpicture}[scale=0.50,every node/.style={transform shape}]
\fill (0,1) circle (1.5pt);
\fill (0,-1) circle (1.5pt);
\fill ({sqrt(3)},0) circle (1.5pt);
\fill (-{sqrt(3)},0) circle (1.5pt);

\node at (0,1.3) {\LARGE{$v_1$}};
\node at (-{sqrt(3)}-0.35,0) {\LARGE{$v_2$}};
\node at (0,-1.3) {\LARGE{$v_3$}};
\node at ({sqrt(3)+0.3},0) {\LARGE{$v_4$}};
%------------------------------
\begin{scope}[decoration={
    markings,
    mark=at position 0.50 with {\arrow{latex}}}
    ] 
\draw[postaction={decorate}] (-{sqrt(3)},0) -- (0,1);
\draw[postaction={decorate}] (-{sqrt(3)},0) -- (0,-1);
\draw[postaction={decorate}] (0,-1) -- (0,1);
\draw[postaction={decorate}] (0,1) -- ({sqrt(3)},0);
\end{scope}

\node at (0,-2) {\Huge{$v_4$}};
\end{tikzpicture}&
%-------09_____v1v4-----------------
\begin{tikzpicture}[scale=0.50,every node/.style={transform shape}]
\fill (0,1) circle (1.5pt);
\fill (0,-1) circle (1.5pt);
\fill ({sqrt(3)},0) circle (1.5pt);
\fill (-{sqrt(3)},0) circle (1.5pt);

\node at (0,1.3) {\LARGE{$v_1$}};
\node at (-{sqrt(3)}-0.35,0) {\LARGE{$v_2$}};
\node at (0,-1.3) {\LARGE{$v_3$}};
\node at ({sqrt(3)+0.3},0) {\LARGE{$v_4$}};
%------------------------------
\begin{scope}[decoration={
    markings,
    mark=at position 0.50 with {\arrow{latex}}}
    ] 
\draw[postaction={decorate}] (0,1) -- (-{sqrt(3)},0);
\draw[postaction={decorate}] (0,-1) -- (-{sqrt(3)},0);
\draw[postaction={decorate}] (0,1) -- (0,-1);
\draw[postaction={decorate}] (0,1) -- ({sqrt(3)},0);
\end{scope}

\node at (0,-2) {\Huge{$v_2v_4$}};
\end{tikzpicture}&
%-------10_____v1-----------------
\begin{tikzpicture}[scale=0.50,every node/.style={transform shape}]
\fill (0,1) circle (1.5pt);
\fill (0,-1) circle (1.5pt);
\fill ({sqrt(3)},0) circle (1.5pt);
\fill (-{sqrt(3)},0) circle (1.5pt);

\node at (0,1.3) {\LARGE{$v_1$}};
\node at (-{sqrt(3)}-0.35,0) {\LARGE{$v_2$}};
\node at (0,-1.3) {\LARGE{$v_3$}};
\node at ({sqrt(3)+0.3},0) {\LARGE{$v_4$}};
%------------------------------
\begin{scope}[decoration={
    markings,
    mark=at position 0.50 with {\arrow{latex}}}
    ] 
\draw[postaction={decorate}] (0,1) -- (-{sqrt(3)},0);
\draw[postaction={decorate}] (0,-1) -- (-{sqrt(3)},0);
\draw[postaction={decorate}] (0,-1) -- (0,1);
\draw[postaction={decorate}] (0,1) -- ({sqrt(3)},0);
\end{scope}

\node at (0,-2) {\Huge{$v_2v_4$}};
\end{tikzpicture}&
%-------11_____v1v4-----------------
\begin{tikzpicture}[scale=0.50,every node/.style={transform shape}]
\fill (0,1) circle (1.5pt);
\fill (0,-1) circle (1.5pt);
\fill ({sqrt(3)},0) circle (1.5pt);
\fill (-{sqrt(3)},0) circle (1.5pt);

\node at (0,1.3) {\LARGE{$v_1$}};
\node at (-{sqrt(3)}-0.35,0) {\LARGE{$v_2$}};
\node at (0,-1.3) {\LARGE{$v_3$}};
\node at ({sqrt(3)+0.3},0) {\LARGE{$v_4$}};
%------------------------------
\begin{scope}[decoration={
    markings,
    mark=at position 0.50 with {\arrow{latex}}}
    ] 
\draw[postaction={decorate}] (0,1) -- (-{sqrt(3)},0);
\draw[postaction={decorate}] (-{sqrt(3)},0) -- (0,-1);
\draw[postaction={decorate}] (0,1) -- (0,-1);
\draw[postaction={decorate}] (0,1) -- ({sqrt(3)},0);
\end{scope}

\node at (0,-2) {\Huge{$v_3v_4$}};

\end{tikzpicture}&
%-------12_____v1-----------------
\begin{tikzpicture}[scale=0.50,every node/.style={transform shape}]
\fill (0,1) circle (1.5pt);
\fill (0,-1) circle (1.5pt);
\fill ({sqrt(3)},0) circle (1.5pt);
\fill (-{sqrt(3)},0) circle (1.5pt);

\node at (0,1.3) {\LARGE{$v_1$}};
\node at (-{sqrt(3)}-0.35,0) {\LARGE{$v_2$}};
\node at (0,-1.3) {\LARGE{$v_3$}};
\node at ({sqrt(3)+0.3},0) {\LARGE{$v_4$}};
%------------------------------
\begin{scope}[decoration={
    markings,
    mark=at position 0.50 with {\arrow{latex}}}
    ] 
\draw[postaction={decorate}] (-{sqrt(3)},0) -- (0,1);
\draw[postaction={decorate}] (-{sqrt(3)},0) -- (0,-1);
\draw[postaction={decorate}] (0,1) -- (0,-1);
\draw[postaction={decorate}] (0,1) -- ({sqrt(3)},0);
\end{scope}

\node at (0,-2) {\Huge{$v_3v_4$}};

\end{tikzpicture}
\end{tabular}
}
    \caption{Acyclic orientations of $G_2$ with corresponding monomials below.} \label{fig:AO_G_2}
\end{figure}

We will show that the acyclic orientation polynomial $A_G(V)$ satisfies a deletion-contraction recurrence. Our theorem generalizes the fact that the number of acyclic orientations satisfies a deletion-contraction recurrence.

For an edge $e\in E$, let $G\setminus e$ be the graph obtained from $G$ by deleting $e$.
When $e$ is not a loop, let $G/e$ be the graph obtained from $G$ by contracting $e$. Any edge in $G$ parallel to $e$ becomes a loop in $G/e$. We denote by $v_e$ the vertex created by contracting $e=\{u_1,u_2\}$, and by $V/e$ the vertex set of $G/e$.
While deleting $e$ leaves the vertex set unchanged, contracting $e$ yields $V/e = V\setminus\{u_1, u_2\}\cup\{v_e\}$ where $u_1$ and $u_2$ are the end vertices of $e$. To make $A_{G/e}(V/e)$ belong to $\ZZ[V]$, we identify
\begin{equation} \label{eq:edge_relation}
v_e=1-(1-u_1)(1-u_2)=u_1+u_2-u_1u_2 \in \ZZ[V],
\end{equation}
and hence we regard $A_{G/e}(V/e)$ as a polynomial in $V$.

\begin{thm}\label{thm:DC}
The acyclic orientation polynomial $A_G(V)$ satisfies the following deletion-contraction recurrence: for every non-loop edge $e\in E$,
\begin{equation}\label{eq:DC}
A_{G}(V)=A_{G\setminus e}(V)+A_{G/e}(V/e).
\end{equation}
\end{thm}
\begin{proof}
Let $e=u_1u_2$ be a non-loop edge of $G$. Fix a non-empty subset $U$ of $V$ and let $a$, $a_D$ and $a_C$ be the coefficients of $\prod_{v \in U}v$ in $A_G(V)$, $A_{G\setminus e}(V)$ and $A_{G/e}(V/e)$, respectively. It is clear that $a=a(G,U)$ and $a_D=a(G\setminus e,U)$. From the relation $v_e=u_1+u_2-u_1u_2$, we see that
$$
a_C =
\begin{cases}
~ -|\mathcal{A}(G/e,U\setminus \{u_1,u_2\}  \cup \{v_e\})|, & \mbox{if } u_1,u_2\in U, \\
~ |\mathcal{A}(G/e,U\setminus \{u_i\}\cup\{v_e\})|, & \mbox{if } U \cap \{u_1,u_2\} =\{u_i\}, \\
~ |\mathcal{A}(G/e,U)|, & \mbox{if } u_1,u_2\notin U.
\end{cases}
$$
To prove equation \eqref{eq:DC}, it suffices to show that
\[
a=a_D+a_C.
\]

Suppose that $U$ contains two end vertices of an edge in $G\setminus e$. Then we have $a=a_D=a_C=0$ since sinks of an acyclic orientation are not adjacent. From now on, we may assume that $U$ has no vertices that are adjacent in $G\setminus e$. We need to consider the following three cases.
\begin{description}
\item[Case 1] $u_1,u_2 \in U$. In this case, $u_1$ and $u_2$ are adjacent in $G$, and hence we have $a=0$. Also merging $u_1$ and $u_2$ gives a bijection between $\mathcal{A}(G\setminus e,U)$ and $\mathcal{A}(G/e,U \setminus \{u_1,u_2\} \cup \{v_e\} )$, which shows that $a_D+a_C=0$.
\item[Case 2] $U \cap \{u_1,u_2\} =\{u_i\}$. Without loss of generality we assume that $u_1\in U$ and $u_2\notin U$. Deleting $e$ induces a bijection between $\mathcal{A}(G,U)$ and $\mathcal{A}(G\setminus e,U)\uplus \mathcal{A}(G\setminus e,U\cup\{u_2\})$:
\[
\mathfrak{o}\in\AO(G,U) \mapsto
\begin{cases}
\mathfrak{o} \setminus e \in\mathcal{A}(G\setminus e,U) , & \mbox{if $u_2$ is not a sink of $\mathfrak{o}\setminus e$,} \\
\mathfrak{o} \setminus e \in\mathcal{A}(G\setminus e,U\cup\{u_2\}) , & \mbox{if $u_2$ is a sink of $\mathfrak{o}\setminus e$.}
\end{cases}
\]
Moreover, since both $u_1$ and $u_2$ do not have outgoing edges in acyclic orientations in $\mathcal{A}(G\setminus e,U\cup\{u_2\})$, we can merge $u_1$ and $u_2$ in $G\setminus e$ to get a one-to-one correspondence between $\mathcal{A}(G\setminus e,U\cup\{u_2\})$ and $\mathcal{A}(G/e,U\setminus \{u_1\}\cup\{v_e\})$. Therefore, we have $a = a_D + |\mathcal{A}(G\setminus e,U\cup\{u_2\})| = a_D + a_C$.
\item[Case 3] $u_1,u_2 \notin U$. Let $\mathcal{A}_{\mbox{$\nleftrightarrow$}}$ be the set of acyclic orientations of $G\setminus e$ which neither contain a directed path from $u_1$ to $u_2$ nor vice versa. Define 
\begin{align*}
\mathcal{A}_1 &= \{\mathfrak{o} \in \mathcal{A}(G,U)\, \mid\,\mathfrak{o} \setminus e \notin \mathcal{A}(G\setminus e,U) \}, \\
\mathcal{A}_2 &= \{\mathfrak{o} \in \mathcal{A}(G,U)\, \mid\,\mathfrak{o} \setminus e \in \mathcal{A}(G\setminus e,U) \mbox{ and } \mathfrak{o} \setminus e \notin \mathcal{A}_{\mbox{$\nleftrightarrow$}} \},  \\
\mathcal{A}_3 &= \{\mathfrak{o} \in \mathcal{A}(G,U)\, \mid\,\mathfrak{o} \setminus e \in \mathcal{A}(G\setminus e,U) \mbox{ and } \mathfrak{o} \setminus e \in \mathcal{A}_{\mbox{$\nleftrightarrow$}} \},  \\
\mathcal{A}_{3,+} &= \{\mathfrak{o} \in \mathcal{A}_3\, \mid\, \overrightarrow{u_1u_2} \in \mathfrak{o} \}, \mbox{ and }\mathcal{A}_{3,-}= \{\mathfrak{o} \in \mathcal{A}_3\, \mid\, \overrightarrow{u_2u_1} \in \mathfrak{o} \}.
\end{align*}
Then 
\[
\mathcal{A}(G,U)=\mathcal{A}_1 \uplus \mathcal{A}_2 \uplus \mathcal{A}_3, \mbox{ and } \mathcal{A}_3=\mathcal{A}_{3,+} \uplus \mathcal{A}_{3,-}.
\]
We claim that $a_D = |\mathcal{A}_2\uplus\mathcal{A}_{3,-}|$ and $a_C = |\mathcal{A}_1\uplus\mathcal{A}_{3,+}|$, which will follow from the fact that the following two maps are bijections:
\begin{align*}
\begin{array}{l}
\mathcal{A}_2 \uplus \mathcal{A}_{3,-} \longrightarrow \mathcal{A}(G\setminus e, U) \\
\quad\quad\quad\,\,\,\,\, \mathfrak{o} \longmapsto \mathfrak{o}\setminus e
  \end{array}
\quad\mbox{and}\quad
\begin{array}{l}
\mathcal{A}_1 \uplus \mathcal{A}_{3,+} \longrightarrow \mathcal{A}(G/e, U) \\
\quad\quad\quad\,\,\,\,\, \mathfrak{o} \longmapsto \mathfrak{o}/e.
  \end{array}
\end{align*}
To verify this fact, we exhibit their inverses below.

First, we define a map from $\mathcal{A}(G\setminus e, U)$ to $\mathcal{A}_2\uplus\mathcal{A}_{3,-}$ as follows: for $\mathfrak{o}\in\mathcal{A}(G\setminus e, U)$,
\[
\mathfrak{o}\mapsto
\begin{cases}
\mathfrak{o}\cup \overrightarrow{u_1u_2} \in\mathcal{A}_2, & \mbox{if $\mathfrak{o}$ contains a directed path from $u_1$ to $u_2$,} \\
\mathfrak{o}\cup \overrightarrow{u_2u_1} \in\mathcal{A}_2, & \mbox{if $\mathfrak{o}$ contains a directed path from $u_2$ to $u_1$,} \\
\mathfrak{o}\cup \overrightarrow{u_2u_1} \in\mathcal{A}_{3,-}, & \mbox{otherwise.}
\end{cases}
\]
Note that by acyclicity of $\mathfrak{o}$, the first and second cases are mutually exclusive. One can check that this map is the inverse of the map $\oo\mapsto \oo\setminus e$ from $\mathcal{A}_2\uplus\mathcal{A}_{3,-}$ to $\mathcal{A}(G\setminus e, U)$.

Similarly, let us construct the inverse of the map $\mathfrak{o}\mapsto\mathfrak{o}/e$ from $\mathcal{A}_1 \uplus \mathcal{A}_{3,+}$ to $\mathcal{A}(G/e, U)$. For $\mathfrak{o}\in\mathcal{A}(G/e, U)$, let $\mathfrak{o}'$ be the acyclic orientation of $G\setminus e$ naturally obtained from $\mathfrak{o}$ by splitting $v_e$ into $u_1$ and $u_2$. Since $v_e$ is not a sink in $\mathfrak{o}$, it follows that $\mathfrak{o}'$ cannot have both $u_1$ and $u_2$ as its sinks at the same time.
Furthermore, there are no paths from $u_1$ to $u_2$ nor vice versa in $\mathfrak{o}'$ because of acyclicity of $\mathfrak{o}$. Define a map from $\mathcal{A}(G/e, U)$ to $\mathcal{A}_1\uplus\mathcal{A}_{3,+}$ as follows:
$$
\mathfrak{o}\mapsto
\begin{cases}
\mathfrak{o}'\cup \overrightarrow{u_1u_2} \in\mathcal{A}_1, & \mbox{if $u_1$ is a sink of $\mathfrak{o}'$,} \\
\mathfrak{o}'\cup \overrightarrow{u_2u_1} \in\mathcal{A}_1, & \mbox{if $u_2$ is a sink of $\mathfrak{o}'$,} \\
\mathfrak{o}'\cup \overrightarrow{u_1u_2} \in\mathcal{A}_{3,+}, & \mbox{otherwise.}
\end{cases}
$$
This map is the inverse of contracting $e$ from $\mathfrak{o}$ in $\mathcal{A}_1 \uplus \mathcal{A}_{3,+}$.
\end{description}
\end{proof}

\begin{ex} \label{ex:SGF}
Let $H$ be the graph whose vertex set is $V(H)=\{v_1,v_2,v_3,v_4\}$ and edge set is $E(H)=\{v_1v_2,v_2v_3,v_3v_1,v_1v_4,v_3v_4\}$. The graphs $H,H\setminus v_3{v_4}$, and $H/v_3{v_4}$ are shown in Figure \ref{fig:G1G2H}. Using the deletion-contraction recurrence together with the computations in Example~\ref{ex:two}, we obtain
\begin{align*}
A_H(V) &= A_{H\setminus v_3{v_4}}(V)+A_{H/v_3{v_4}}(\{v_1,v_2,v_{34}=v_3+v_4-v_3v_4\}) \\
&= 2(v_1+v_2+v_3+v_4+v_2v_4+v_3v_4)+2(v_1+v_2+v_3+v_4-v_3v_4)
\\
&= 4(v_1+v_2+v_3+v_4)+2v_2v_4, 
\end{align*}
and hence we have $a_H(t) = 16t+2t^2$.
\end{ex}

\begin{figure}[h]
\centering
\resizebox{.9\linewidth}{!}{
\setlength{\tabcolsep}{25pt}
    \begin{tabular}{ccc}

\begin{tikzpicture}[scale=0.8,every node/.style={transform shape}] \fill (0,1) circle (2pt); \fill (0,-1) circle (2pt); \fill ({sqrt(3)},0) circle (2pt); \fill (-{sqrt(3)},0) circle (2pt); \node at (0,1.3) {\Large{$v_1$}}; \node at (-{sqrt(3)}-0.35,0) {\Large{$v_2$}}; \node at (0,-1.3) {\Large{$v_3$}}; \node at ({sqrt(3)+0.3},0) {\Large{$v_4$}}; 
%------------------------------ 
\draw (0,1) -- (-{sqrt(3)},0); \draw (-{sqrt(3)},0) -- (0,-1); \draw (0,1) -- (0,-1); \draw ({sqrt(3)},0) --(0,1); \draw ({sqrt(3)},0) --(0,-1); 
\node at (0,-2) {\large $H$};
\end{tikzpicture} 
&
\begin{tikzpicture}[scale=0.8,every node/.style={transform shape}]
 \fill (0,1) circle (2pt); \fill (0,-1) circle (2pt); \fill ({sqrt(3)},0) circle (2pt); \fill (-{sqrt(3)},0) circle (2pt); \node at (0,1.3) {\Large{$v_1$}}; \node at (-{sqrt(3)}-0.35,0) {\Large{$v_2$}}; \node at (0,-1.3) {\Large{$v_3$}}; \node at ({sqrt(3)+0.3},0) {\Large{$v_4$}}; %------------------------------
 
\draw (0,1) -- (-{sqrt(3)},0); \draw (-{sqrt(3)},0) -- (0,-1); \draw (0,1) -- (0,-1); \draw ({sqrt(3)},0) --(0,1); 
\node at (0,-2) {\large $H\setminus v_3v_4$};
\end{tikzpicture}
&
\begin{tikzpicture}[scale=0.8,every node/.style={transform shape}] \fill (0,1) circle (2pt); \fill (0,-1) circle (2pt); \fill (-{sqrt(3)},0) circle (2pt); \node at (0,1.3) {\Large{$v_1$}}; \node at (-{sqrt(3)}-0.355,0) {\Large{$v_2$}}; \node at (0,-1.3) {\Large{$v_{34}$}}; \node at ({sqrt(3)+0.3},0) {\quad\quad}; 
%------------------------------ 
\draw (0,1) -- (-{sqrt(3)},0); \draw (-{sqrt(3)},0) -- (0,-1); \draw (0,1) to [bend left] (0,-1); \draw (0,1) to [bend right] (0,-1);
\node at (0,-2) {\large $H/v_3v_4$};
\end{tikzpicture}
    \end{tabular}
}
    \caption{Graphs $H, H\setminus v_3v_4$, and $H/v_3v_4$.} \label{fig:G1G2H}
\end{figure}
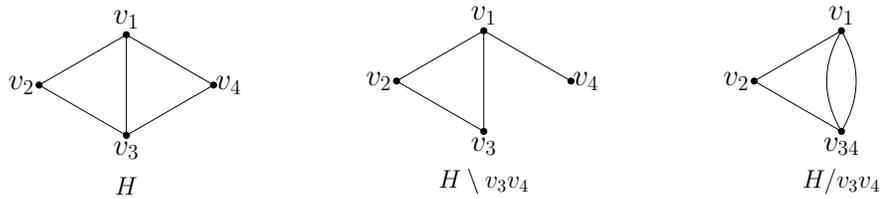 

We remark that this deletion-contraction recurrence gives a connection between the acyclic orientation polynomial and the \emph{$\VV$-polynomial}; see Appendix~\ref{sec:appendix}.

\section{Four expressions for acyclic orientation polynomials}\label{sec:Expression}
%%%%%%%%%%%%%%%%%%%%%%%%%%%%%%%%%%%%%%%%%%%%%%%%%%%%%%%%%%%%%%%%%%%%%%%%%%%%%%%
\subsection{Subgraph expansions}\label{sec:Sub} We will expand our acyclic orientation polynomial $A_G(V)$ with respect to spanning subgraphs of $G$. Let $S$ be a subset of the edge set $E$. The set $S$ will be identified with the spanning subgraph of $G$ whose edge set is $S$. Denote by $|S|$ the number of edges of $S$. Let $\mathcal{S}(G)$ be the collection of spanning subgraphs of $G$. For an edge $e=u_1u_2 \in E(G)$, define
\[
\mathcal{S}(G)^e=\{S \in \mathcal{S}(G)\, \mid\, e \notin E(S) \}, \mbox{ and } \mathcal{S}(G)_e=\mathcal{S}(G) \setminus \mathcal{S}(G)^e. 
\]
Note that a subgraph in $\mathcal{S}(G)_e$ has a connected component which contains both $u_1$ and $u_2$.

\begin{prop}\label{prop:AC_Sub}
For an edge $e \in E(G)$, $\mathcal{S}(G)^e=\mathcal{S}(G\setminus e)$ and there is a bijection between $\mathcal{S}(G)_e$ and $\mathcal{S}(G/e)$. 
\end{prop}
\begin{proof}
Straightforward from the definitions.
\end{proof}

Let $\mathcal{C}(S)$ be the set of connected components of a subgraph $S$. For each connected component $C \in \mathcal{C}(S)$, define
\[
v_C=1-\prod_{v \in V(C)}(1-v) \in\ZZ[V].
\]
Note that if $C$ is a graph consisting of two vertices joined by an edge $e$, then $v_C = v_e$. Let $s(S)$ be the \emph{corank} of $S$ defined as $s(S)=|S|-d+|\mathcal{C}(S)|$. 

\begin{thm}\label{thm:AC_Sub}
The acyclic orientation polynomial of $G=(V,E)$ is given by
\[
A_G(V)=\sum_{S \subseteq E}(-1)^{s(S)}\prod_{C \in \mathcal{C}(S)}v_C.
\]
Hence,
\begin{equation}\label{eq2:Sub}
a_G(t)=\sum_{S \subseteq E}(-1)^{s(S)}{\prod_{C \in \mathcal{C}(S)}{\big(1-(1-t)^{|V(C)|}\big)}}. 
\end{equation}
\end{thm}
\begin{proof}
We prove the theorem by induction on the number of edges, the base case that $G$ has no edges being trivial. Take a non-loop edge $e=u_1u_2 \in E$. Thanks to the deletion-contraction recurrence \eqref{eq:DC} for $A_G(V)$, it is enough to show that
\[
A_{G\setminus e}(V) = \sum_{S\in\mathcal{S}(G)^e} (-1)^{s(S)}\prod_{C \in \mathcal{C}(S)}v_C \quad\mbox{and}\quad A_{G/e}(V/e) = \sum_{S\in\mathcal{S}(G)_e} (-1)^{s(S)}\prod_{C \in \mathcal{C}(S)}v_C.
\]
The first equation follows immediately from Proposition~\ref{prop:AC_Sub} and the induction hypothesis.
To verify the second equation, let $S\in\mathcal{S}(G)_e$ be a subgraph of $G$ containing $e$.
Since contracting $e$ leaves connected components not containing $e$ unchanged, we only consider the connected component $C_e$ of $S$ containing $e$.
Let $S'\in\mathcal{S}(G/e)$ be the subgraph of $G/e$ obtained from $S$ by contracting $e$ and $C'_e$ the connected component containing $v_e$.
Using the relation $1-v_e = (1-u_1)(1-u_2)$ and the fact that $V(C'_e)\setminus\{v_e\}\cup\{u_1,u_2\} = V(C_e)$, we obtain $v_{C_e} = v_{C'_e}$. Since contracting $e$ preserves the corank of $S$, we deduce
\[
(-1)^{s(S)}\prod_{C \in \mathcal{C}(S)}v_C = (-1)^{s(S')}\prod_{C' \in \mathcal{C}(S')}v_{C'},
\]
and then the proof follows from Proposition~\ref{prop:AC_Sub} and the induction hypothesis.
\end{proof}

\begin{ex} Let $P$ be the path graph of length 2. See Figure~\ref{fig:path}. We can compute $\prod_{C \in \mathcal{C}(S)}v_C$ for each spanning subgraph $S$. For instance, let $S=\{v_1 v_2\}$ be the second subgraph in Figure~\ref{fig:path}. The graph $S$ has two connected components whose vertex sets are $\{v_1,v_2\}$ and $\{v_3\}$. Then the corresponding term is
\[
\prod_{C \in \mathcal{C}(S)}v_C = (1-(1-v_1)(1-v_2))v_3.
\]
Note that the corank of each subgraph of $P$ is equal to 0. Using the previous theorem, we have
\begin{align*}
A_P(V) &= (1-(1-v_1)(1-v_2)(1-v_3)) + (1-(1-v_1)(1-v_2))v_3 \\
&\quad + (1-(1-v_2)(1-v_3))v_1 + v_1 v_2 v_3 \\
 &= v_1 + v_2 + v_3 + v_1 v_3.
\end{align*}
\begin{figure}[h]
\centering
\resizebox{\linewidth}{!}{
\begin{tabular}{cccc}
%-------S={12,23}-----------------
\begin{tikzpicture}[scale=0.75,every node/.style={transform shape}]
\fill (0,1) circle (1.5pt);
\fill ({sqrt(3)},0) circle (1.5pt);
\fill (-{sqrt(3)},0) circle (1.5pt);

\node at (0,1.3) {\Large{$v_2$}};
\node at (-{sqrt(3)}-0.355,0) {\Large{$v_1$}};
\node at ({sqrt(3)+0.35},0) {\Large{$v_3$}};

\draw (-{sqrt(3)},0) -- (0,1);
\draw ({sqrt(3)},0) -- (0,1);

\node[scale=1.4] at (0,-1) {$1-(1-v_1)(1-v_2)(1-v_3)$};
\end{tikzpicture}&
%-------S={12}-----------------
\begin{tikzpicture}[scale=0.75,every node/.style={transform shape}]
\fill (0,1) circle (1.5pt);
\fill ({sqrt(3)},0) circle (1.5pt);
\fill (-{sqrt(3)},0) circle (1.5pt);

\node at (0,1.3) {\Large{$v_2$}};
\node at (-{sqrt(3)}-0.355,0) {\Large{$v_1$}};
\node at ({sqrt(3)+0.35},0) {\Large{$v_3$}};

\draw (-{sqrt(3)},0) -- (0,1);
% \draw ({sqrt(3)},0) -- (0,1);

\node[scale=1.4] at (0,-1) {$(1-(1-v_1)(1-v_2))v_3$};
\end{tikzpicture}&
%-------S={23}-----------------
\begin{tikzpicture}[scale=0.75,every node/.style={transform shape}]
\fill (0,1) circle (1.5pt);
\fill ({sqrt(3)},0) circle (1.5pt);
\fill (-{sqrt(3)},0) circle (1.5pt);

\node at (0,1.3) {\Large{$v_2$}};
\node at (-{sqrt(3)}-0.355,0) {\Large{$v_1$}};
\node at ({sqrt(3)+0.35},0) {\Large{$v_3$}};

% \draw (-{sqrt(3)},0) -- (0,1);
\draw ({sqrt(3)},0) -- (0,1);

\node[scale=1.4] at (0,-1) {$(1-(1-v_2)(1-v_3))v_1$};
\end{tikzpicture}&
%-------S={}-----------------
\begin{tikzpicture}[scale=0.75,every node/.style={transform shape}]
\fill (0,1) circle (1.5pt);
\fill ({sqrt(3)},0) circle (1.5pt);
\fill (-{sqrt(3)},0) circle (1.5pt);

\node at (0,1.3) {\Large{$v_2$}};
\node at (-{sqrt(3)}-0.355,0) {\Large{$v_1$}};
\node at ({sqrt(3)+0.35},0) {\Large{$v_3$}};

% \draw (-{sqrt(3)},0) -- (0,1);
% \draw ({sqrt(3)},0) -- (0,1);

\node[scale=1.4] at (0,-1) {$v_1 v_2 v_3$};
\end{tikzpicture}
\end{tabular}
}
    \caption{All spanning subgraphs of $P$ and their corresponding terms.} \label{fig:path}
\end{figure}
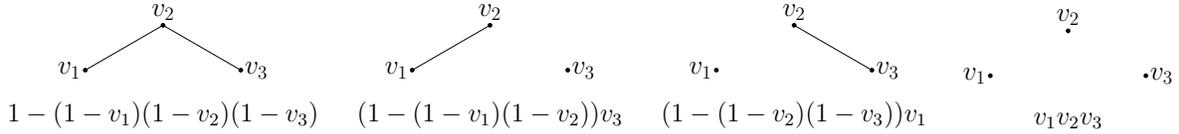
\end{ex}

%%%%%%%%%%%%%%%%%%%%%%%%%%%%%%%%%%%%%%%%%%%%%%%%%%%%%%%%%%%%%%%%%%%%%%%%%%%%%%%
\subsection{Broken circuit complexes}\label{sec:AC_BC}
We will give an expression for $A_G(V)$ in terms of the broken circuit complex $\BB_G$ as an analogue of Whitney's theorem \cite{Whi32}.

Let the edge set $E$ be linearly ordered. A \emph{broken circuit} is a cycle with its smallest edge removed.
The \emph{broken circuit complex} $\BB_G$ of $G$ is the collection of all spanning subgraphs $S$ which do not contain a broken circuit. For a non-loop edge $e \in E$, define
\[
(\BB_G)^e=\{S \in \BB_G\, \mid\, e \notin E(S) \}, \mbox{ and } (\BB_G)_e=\BB_G \setminus (\BB_G)^e. 
\]
The following proposition appears in \cite{BS86} and \cite[Lemma 3.7]{GS01}. 
\begin{prop}\label{prop:AC_BC} 
If $e$ is the largest edge in $E(G)$, then $(\BB_G)^e=\BB_{G\setminus e}$ and there is a bijection between $(\BB_G)_e$ and $\BB_{G / e}$. 
\end{prop}

\begin{thm}\label{thm:AC_BC} The acyclic orientation polynomial of $G=(V,E)$ is given by
\begin{equation}\label{eq:AC_and_BC}
A_G(V)=\sum_{S \in \BB_G}\prod_{C \in \mathcal{C}(S)}v_C.
\end{equation}
\end{thm}
\begin{proof}
The proof is parallel to that of Theorem~\ref{thm:AC_Sub}, so we only sketch the proof.

We induct on the number of edges. Let $e=u_1 u_2$ be a non-loop edge. By the definitions of $(\BB_G)^e$ and $(\BB_G)_e$, the right-hand side of equation \eqref{eq:AC_and_BC} is
\[
\sum_{S \in (\BB_G)^e}\prod_{C \in \mathcal{C}(S)}v_C + \sum_{S \in (\BB_G)_e}\prod_{C \in \mathcal{C}(S)}v_C.
\]
From the deletion-contraction recurrence for $A_G(V)$, it suffices to show that
\[
A_{G\setminus e}(V) = \sum_{S \in (\BB_G)^e}\prod_{C \in \mathcal{C}(S)}v_C, \qand A_{G/e}(V/e) = \sum_{S \in (\BB_G)_e}\prod_{C \in \mathcal{C}(S)}v_C.
\]
The first equation follows from Proposition~\ref{prop:AC_BC}, and the second equation can be verified by a similar argument in the proof of Theorem~\ref{thm:AC_Sub}.
\end{proof}

The linear terms of equation \eqref{eq:AC_and_BC} give the following corollary.
\begin{cor}[{\cite[Theorem 7.3]{GZ83}}]\label{cor:unique}
Let $G=(V,E)$ be a connected graph on $d$ vertices. For any $v\in V$, the number of acyclic orientations of $G$ with a unique sink $v$ is equal to the number of subsets of $E$ of size $d-1$ containing no broken circuit.
\end{cor}
\begin{proof} For $S \in \BB_G$, the graph $S$ is connected if and only if $|S|=d-1$. Since $v_C$ has no constant terms for a connected component $C \in \mathcal{C}(S)$, the degree of each term of $\prod_{C \in \mathcal{C}(S)}v_C$ is greater than $1$ if $|S|<d-1$. Thus, by equation \eqref{eq:AC_and_BC}, the linear terms of $A_G(V)$ are the same as those of   
\[
\sum_{\substack{S \in \BB_G \\|S|=d-1}}v_S = \sum_{\substack{S \in \BB_G \\|S|=d-1}}\left(1-\prod_{v\in V}(1-v)\right).
\]
Therefore, the coefficient of $v \in V$ in $A_G(V)$ is equal to $|\{S \in \BB_G \mid |S|=d-1\}|$.

\end{proof}

\subsection{Bond lattices}\label{sec:AC_LG}
We will express $A_G(V)$ in terms of the bond lattice $L_G$ as an analogue of Birkhoff's Theorem \cite{Bir12}. In this and next subsections, we assume that $G$ is a simple graph, i.e., $G$ has no multiple edges or loops.

We assume that the reader is familiar with the basics of the theory of posets, including the definition of a lattice. A good reference is \cite[Chapter 3]{Sta11}. We only review the M\"obius function.
Let $P$ be a finite poset and $<$ denote the partial order on $P$. The \emph{M\"obius function} $\mu_P$ of $P$ is a map $P \times P \rightarrow \ZZ$ given by
\begin{align*}
\mu_P(s,u) =
\begin{dcases*}
1, & if $s=u$, \\
-\sum_{s\le t<u} \mu_P(s,t), & if $s<u$, \\
0, & otherwise.
\end{dcases*}
\end{align*}
We will use the following proposition, called the \emph{M\"obius inversion formula}; see \cite[Proposition 3.7.1]{Sta11}.
\begin{prop} \label{prop:mobius_inv}
Let $P$ be a finite poset and $R$ be a commutative ring with unity. Let $f,g: P\rightarrow R$. Then
\[
g(s) = \sum_{t\ge s} f(t), \mbox{ for all $s\in P$,}
\]
if and only if
\[
f(s) = \sum_{t\ge s} \mu_P(s,t) g(t), \mbox{ for all $s\in P$.}
\]
\end{prop}
For the M\"obius function of a lattice, the following theorem is well-known; see \cite[Corollary (b), p. 351]{ROT64} for the proof.
\begin{thm}[Weisner's theorem] \label{thm:weisner}
Let $L$ be a finite lattice with at least two elements, and $\hat{1}_L$ its greatest element. For $a \in L$ with $a \neq \hat{1}_L$ and any $b\in L$,
\[
\sum_{x\wedge a=b} \mu_L(x,\hat{1}_L) = 0,
\]
where $s\wedge t$ denotes the largest element $p$ satisfying $p \leq s$ and $p \leq t$.
\end{thm}

Let us now return to the graph theory.
A \emph{vertex partition} $\pi$ is a set of mutually disjoint non-empty subsets of $V$ whose union is $V$. Each $B\in\pi$ is called a \emph{block}.
A \emph{bond} is a vertex partition each of whose blocks induces a connected graph. The set $L_G$ of bonds of $G$ forms a poset partially ordered by refinement. In fact, $L_G$ is a lattice, not just a poset, and we call $L_G$ the \emph{bond lattice} of $G$.
The least element $\hat{0}$ of $L_G$ is the bond each of whose blocks consists of a single vertex, and the greatest element $\hat{1}$ of $L_G$ is the bond each of whose blocks is the vertex set of a connected component of $G$. We write $\mu_G$, instead of $\mu_{L_G}$, for the M\"obius function of $L_G$.
The \emph{M\"obius invariant} of $G$ is defined as $\mu(G)=|\mu_G(\hat{0},\hat{1})|$. Note that $|\mu_G(\hat{0},\pi)|=(-1)^{d-|\pi|}\mu_G(\hat{0},\pi)$ for $\pi \in L_G$. Greene and Zaslavsky~\cite[Theorem~7.3]{GZ83} proved the following theorem via the theory of hyperplane arrangements, and Gebhard and Sagan~\cite{GS00} presented three elementary proofs of this theorem.
\begin{thm} \label{thm:unique2}
Let $G=(V,E)$ be a connected graph. For any $v \in V$, we have $a(G,v)=\mu(G)$.
\end{thm}

Let us take a bond $\pi \in L_G$. Let $G/\pi$ be the graph obtained from $G$ by contracting each block $B \in \pi$ to the vertex $v_B$, and denote by $V/\pi$ the vertex set of $G/\pi$. For each $B \in \pi$, to the vertex $v_B$ we associate the element of $\ZZ[V]$ defined by
\begin{equation}  \label{eq:v_B}
v_B = 1-\prod_{v \in B}(1-v).
\end{equation}
\begin{thm}\label{thm:AC_LG}
The acyclic orientation polynomial of $G=(V,E)$ is given by
\begin{equation}\label{eq:Mobius}
A_G(V)=\sum_{\pi \in L_G}(-1)^{d-|\pi|}\mu_G{(\hat{0},\pi)}\prod_{B \in \pi}{v_B}.
\end{equation}
\end{thm}

To prove Theorem~\ref{thm:AC_LG}, we need the following key proposition. For a subset $U\subseteq V$, we denote by $a(G,\subseteq U)$ the number of acyclic orientations of $G$ whose sinks belong to $U$.
\begin{prop} \label{prop:lattice}
For a proper subset $U$ of $V$, we have
\begin{equation} \label{eq:alter_sum_lattice_aG/pi}
\sum_{\pi\in L_G} (-1)^{d-|\pi|}a(G/\pi, \subseteq U/\pi) = 0,
\end{equation}
where $U/\pi$ is the subset of $V/\pi$ corresponding to $U$, that is, $U/\pi = \{ v_{B} \in V/\pi \mid B\cap U \neq \emptyset\}$.
\end{prop}
\begin{proof}
When $U$ is empty, there is nothing to show, so we suppose that $U$ is a non-empty subset of $V$.

Let $G^U$ be the graph obtained from $G$ by adding an extra vertex $u_0$ and joining $u_0$ and $v$ for each $v\in U$. In other words, $G^U$ is the graph whose vertex set is $V\cup \{u_0\}$ and edge set is $E\cup \{u_0v \mid v\in U\}$.
For $\mathfrak{o} \in \mathcal{A}(G^U,\{u_0\})$, deleting $u_0$ yields an acyclic orientation of $G$ whose sinks are contained in $U$. This procedure is bijective, which gives the following identity:
\begin{equation}\label{eq:unique_and_union}
a(G^U, u_0)=a(G, \subseteq U).
\end{equation}

Define the map $\iota:L_G\to L_{G^U}$ by $\iota(\pi)=\pi\cup\{\{u_0\}\}$ for each $\pi \in L_G$. Thus the lattice $L_G$ is embedded in $L_{G^U}$ via this map. The greatest element of $L_{G^U}$ will be denoted by $\hat{1}_{G^U}$.

For a bond $\pi\in L_G$, we have
\begin{align*}
a(G/\pi, \subseteq U/\pi)
&= a((G/\pi)^{U/\pi}, u_0) && \mbox{by equation~\eqref{eq:unique_and_union},} \\
&= a(G^U/\iota(\pi), u_0) && \mbox{by $(G/\pi)^{U/\pi} \simeq G^U/\iota(\pi)$,} \\
&= \mu(G^U/\iota(\pi)) && \mbox{by Theorem~\ref{thm:unique2}.}
\end{align*}
Since the bond lattice of $G^U/\iota(\pi)$ is isomorphic to the interval $[\iota(\pi), \hat{1}_{G^U}]$ in $L_{G^U}$ and $|\iota(\pi)|=|\pi|+1$, we deduce the following identity:
\[
a(G/\pi, \subseteq U/\pi) = (-1)^{|\pi|} \mu_{G^U}(\iota(\pi), \hat{1}_{G^U}).
\]
Using this identity, we can rewrite equation \eqref{eq:alter_sum_lattice_aG/pi} as
\begin{equation} \label{eq:rewrite_lattice}
\sum_{\pi\in L_G} \mu_{{G^U}}(\iota(\pi), \hat{1}_{G^U})=0.
\end{equation}
Denote by $\widetilde{\pi}_0$ the bond in $L_{G^U}$ consisting of $U\cup \{u_0\}$ and $d-|U|$ singletons. Note that $\widetilde{\pi}_0\neq \hat{1}_{L^G}$ because $U$ is a proper subset of $V$.
Let $\mathcal{M}=\{{\iota(\pi)\wedge\widetilde{\pi}_0} \in L_{G^U}\mid\pi\in L_{G}\}$. Then $L_G$ can be partitioned as $L_G = \cup_{\widetilde{\pi}\in \mathcal{M}} \{\pi\in L_G\mid \iota(\pi)\wedge \widetilde{\pi}_0 = \widetilde{\pi} \}$, which shows that the left-hand side of equation \eqref{eq:rewrite_lattice} is 
\begin{equation}\label{eq:inner}
\sum_{\widetilde{\pi}\in \mathcal{M}} \sum_{\substack{\pi \in L_G \\\iota(\pi)\wedge \widetilde{\pi}_0 = \widetilde{\pi}}} \mu_{{G^U}}(\iota(\pi), \hat{1}_{G^U}).
\end{equation}
For $\widetilde{\pi}\in \mathcal{M}$, one can easily check that for an element $\pi' \in L_{G^U}$ with $\pi'\wedge \widetilde{\pi}_0 = \widetilde{\pi}$, there exists a unique $\pi \in L_G$ such that $\pi'=\iota(\pi)$. Then we apply Theorem~\ref{thm:weisner} to the inner summation in \eqref{eq:inner}, and hence we obtain the desired result \eqref{eq:rewrite_lattice}.
\end{proof}

\begin{proof}[Proof of Theorem~\ref{thm:AC_LG}]
For each $\pi\in L_G$, let $f(\pi) = (-1)^{|\pi|} A_{G/\pi}(V/\pi)$ and $g(\pi) = (-1)^{|\pi|} {\prod_{B\in\pi} v_B}$. Using the assignment \eqref{eq:v_B}, we regard $f(\pi)$ and $g(\pi)$ as elements of $\ZZ[V]$. Then we can restate equation \eqref{eq:Mobius} as
\[
f(\hat{0}) = \sum_{\pi\in L_G} \mu_G(\hat{0},\pi) g(\pi).
\]
By the M\"obius inversion formula, it suffices to show that for each $\pi\in L_G$,
\[
g(\pi) = \sum_{\pi'\ge \pi} f(\pi'),
\]
or equivalently,
\begin{equation} \label{eq:Mobius2}
\prod_{B\in\pi} v_B = \sum_{\pi'\ge\pi} (-1)^{|\pi'|-|\pi|} A_{G/\pi'}(V/\pi').
\end{equation}

We will first show the identity~\eqref{eq:Mobius2} in the case where $\pi=\hat{0}$, and then explain how this proof implies a proof for a general case. For the case where $\pi=\hat{0}$, we need to show that
\begin{equation} \label{eq:Mobius2_base}
\prod_{v\in V} v = \sum_{\pi'\in L_G} (-1)^{d-|\pi'|} A_{G/\pi'}(V/\pi').
\end{equation}
One can easily see that $A_{G/\pi'}(V/\pi')$ is a square-free polynomial in $\ZZ[V]$ for each $\pi'\in L_G$. Then we write
\begin{equation} \label{eq:square-free_poly}
\sum_{\pi'\in L_G} (-1)^{d-|\pi'|} A_{G/\pi'}(V/\pi') = \sum_{U\subseteq V} c_U \prod_{v\in U} v.
\end{equation}
For a proper subset $U\subsetneq V$, let us evaluate the left-hand side of equation~\eqref{eq:square-free_poly} at $v=1$ if $v\in U$, and $v=0$ otherwise. From the definition of $v_B$, we see that $v_B=1$ if $B\cup U \neq \emptyset$, and $v_B=0$ otherwise. Then the quantity $A_{G/\pi'}(V/\pi')$ with this evaluate is equal to $a(G/\pi', \subseteq U/\pi')$, and therefore Proposition~\ref{prop:lattice} says that the left-hand side of equation~\eqref{eq:square-free_poly} with this evaluation equals 0.

Hence the values $c_U$ must be $0$ for all $U \subsetneq V$. If not, there exists $U \subsetneq V$ such that $c_U\ne 0$ and $c_W = 0$ for all $W\subsetneq U$. By the evaluation in the previous paragraph, the right-hand side of equation~\eqref{eq:square-free_poly} is $\sum_{W \subseteq U}{c_W}=0$. Thus, we have $c_U=0$, which is a contradiction.

We claim that $c_V=1$. Since $v_{B}$ is a polynomial of degree $|B|$, the degree of the polynomial $A_{G/\pi}(V/\pi)$ is less than $|V|$ if $G/\pi$ has an edge. So the polynomial $A_{G/\pi'}(V/\pi')$ where $\pi'$ is the greatest element $\hat{1}$ of $L_G$ is the only term which can contribute to $c_V$, and the coefficient of $\prod_{v \in V}v$ in $A_{G/\pi'}(V/\pi')$ is $(-1)^{d-|\pi'|}$, which proves the claim, and the identity~\eqref{eq:Mobius2_base}.

Now fix an arbitrary bond $\pi$ in $L_G$, and take a bond $\pi'\in L_G$ with $\pi'>\pi$. Let $\varphi_\pi:\ZZ[V/\pi]\rightarrow \ZZ[V]$ be the map given by
\[
v_B \longmapsto 1-\prod_{v\in B}(1-v),
\]
and $\varphi_{\pi'}$ similarly.
Also we can think of $\pi'$ as an element of $L_{G/\pi}$, and then define $\varphi_{\pi',\pi}:\ZZ[V/\pi']\rightarrow \ZZ[V/\pi]$ similarly. Then it is easy to check that the following diagram commutes:
\[
\begin{tikzcd}[column sep=1.5em]
\ZZ[V/\pi'] \arrow[dr, "\varphi_{\pi'}"] \arrow[rr, "\varphi_{\pi',\pi}"] && \ZZ[V/\pi] \arrow[dl, "\varphi_{\pi}"] \\
 & \ZZ[V]
\end{tikzcd}
\]

Applying the identity~\eqref{eq:Mobius2_base} to the graph $G/\pi$, we have
\[
\prod_{B\in\pi} v_B = \sum_{\pi'\in L_{G/\pi}} (-1)^{|\pi'|-|\pi|} A_{(G/\pi)/\pi'}(V/\pi')
\]
as polynomials in $\ZZ[V/\pi]$. Applying $\varphi_{\pi}$ on both sides completes the proof since $L_{G/\pi}\simeq [\pi,\hat{1}]$ in $L_G$, and $\varphi_{\pi}\circ\varphi_{\pi',\pi}=\varphi_{\pi'}$.
\end{proof}

Using Theorem~\ref{thm:AC_LG}, we can express $a(G,U)$ in terms of the M\"obius function.
For a subset $U$ of $V$, define 
\[
\mathcal{R}(U)=\{\pi \in L_G \mid B \cap U \ne \emptyset \mbox{ for every } B \in \pi \}.
\]
Extracting the coefficient of $\prod_{v \in U}v$ in equation \eqref{eq:Mobius} yields
\begin{equation}\label{eq:a(G,U)}
a(G,U) = \sum_{\pi \in \mathcal{R}(U)} (-1)^{d-|U|}\mu_G(\hat{0}, \pi).
\end{equation}

\subsection{The map from $\mathcal{A}(G)$ to $L_G$}\label{sec:AC_SD}
In this section, we recall a map from $\mathcal{A}(G)$ to $L_G$ which appeared in \cite[Section 4]{BN20} to obtain an expression of $a(G, \subseteq U)$. After that, we present another proof of Theorem~\ref{thm:AC_LG} using the inclusion-exclusion principle, and give an another expansion of $A_G(V)$ which is an analogue of \cite[Corollary 7.4]{GZ83}.

Let us fix an ordering of $V$, and take an acyclic orientation $\mathfrak{o} \in \mathcal{A}(G)$. For $j \ge 1$, suppose that $B_1, B_2, \dots, B_{j-1} \subseteq V$ are defined.
Denote by $s_j$ the smallest element in $V\setminus \bigcup_{i=1}^{j-1} B_i$. Let $B_j$ be the collection of vertices in $V\setminus\bigcup_{i=1}^{j-1} B_i$ reachable to $s_j$ via a directed path in $\mathfrak{o}$. Define $\pi(\mathfrak{o})=\{B_1, \cdots, B_q\}$, where $q$ is the largest integer with $B_{q} \ne \emptyset$.
Clearly, $\pi(\mathfrak{o})\in L_G$. The blocks in $\pi(\mathfrak{o})$ are called the \emph{sink-components} and the map $\pi\colon \mathcal{A}(G) \to L_G$ is called the \emph{sink-component map} with respect to the ordering on vertices.
(This map connecting bonds and acyclic orientations of a given graph was first introduced by Greene and Zaslavsky \cite{GZ83}. However, their original definition was flawed, invalidating the proofs of the results in \cite[Section 7]{GZ83}. With this corrected definition, their theorems are now valid.)

The following theorem gives a relation between the sink-component map and the M\"obius function.
\begin{thm}[{\cite[Theorem 7.4]{GZ83}}]\label{thm:7.4} For a bond $\pi \in L_G$, the cardinality of the preimage $\{ \mathfrak{o} \in \mathcal{A}(G) \mid \pi(\oo)=\pi \}$ under the sink-component map is equal to $|\mu_G(\hat{0},\pi)|=(-1)^{d-|\pi|}\mu(\hat{0},\pi)$.
\end{thm}
Using this theorem, we can derive the following equation for $a(G, \subseteq U)$.
\begin{prop}\label{prop:inclusion}
For a subset $U\subseteq V$, $a(G,\subseteq U) = \sum_{\pi \in \mathcal{R}(U)} (-1)^{d-|\pi|}\mu(\hat{0},\pi)$.
\end{prop}
\begin{proof}
Thanks to Theorem~\ref{thm:7.4}, it suffices to show that
\begin{align*}
\mathfrak{o} \in a(G, \subseteq U) \text{ if and only if } \pi(\mathfrak{o}) \in \mathcal{R}(U).
\end{align*}

Consider an ordering on $V$ with the property that each element in $U$ is smaller than any element in $V\setminus U$, and let $\pi\colon \mathcal{A}(G) \to L_G$ be the sink-component map with respect to the ordering.
If there is a sink $v$ of $\mathfrak{o}$ which is not in $U$, then the sink-component containing $v$ does not intersect $U$, which implies that $\pi(\mathfrak{o}) \notin \mathcal{R}(U)$.

Conversely, suppose that $\pi(\mathfrak{o}) \notin \mathcal{R}(U)$. Let $B$ be the last sink-component of $\pi(\oo)$, that is, the smallest vertex $s$ in $B$ is the largest one among the smallest vertices of each block of $\pi(\oo)$. Then by the ordering of $V$, $B$ and $U$ do not intersect.
We claim that $s$ is a sink of $\mathfrak{o}$, which implies that $\mathfrak{o}\notin a(G,\subseteq U)$. If $s$ is not a sink, then there is an edge directed from $s$ to a vertex in some other sink-component by acyclicity of $\mathfrak{o}$. Then $s$ must be in that sink-component which is a contradiction.
\end{proof}
Now we present an alternative proof of Theorem~\ref{thm:AC_LG} using the inclusion-exclusion principle.
\begin{proof}[Second Proof of Theorem~\ref{thm:AC_LG}]
We will prove the identity~\eqref{eq:a(G,U)}, which is equivalent to Theorem~\ref{thm:AC_LG}.
Let $U$ be a subset of $V$, then we have
\begin{align*}
a(G, U) & =\sum_{\substack{W \subseteq U}} (-1)^{|U|-|W|} a(G,\subseteq W) && \mbox{by the inclusion-exclusion principle,} \\
& =\sum_{\substack{W \subseteq U}} (-1)^{|U|-|W|} \sum_{\pi \in \mathcal{R}(W)} (-1)^{d-|\pi|}\mu(\hat{0},\pi) && \mbox{by Proposition~\ref{prop:inclusion},} \\
& =\sum_{\substack{\pi \in \mathcal{R}(U)}} (-1)^{|U|+d-|\pi|} \mu(\hat{0},\pi) \sum_{\substack{W \subseteq U \\ \pi \in \mathcal{R}(W)}} (-1)^{|W|}.
\end{align*}
The inner summation in the right-hand side of the last equation can be simplified as follows: for $\pi\in\RR(U)$,
\[
\sum_{\substack{W \subseteq U \\ \pi \in \mathcal{R}(W)}}
(-1)^{|W|}=\prod_{B \in \pi} \Big( \sum_{\substack{B' \subseteq B\cap U \\ B' \neq \emptyset}}(-1)^{|B'|}\Big) =(-1)^{|\pi|}.
\]
Therefore, we deduce that
\begin{equation*}
a(G, U)
=\sum_{\substack{\pi \in \mathcal{R}(U)}} (-1)^{|U|+d-|\pi|} \mu(\hat{0},\pi) (-1)^{|\pi|}
=\sum_{\substack{\pi \in \mathcal{R}(U)}} (-1)^{d-|U|} \mu(\hat{0},\pi),
\end{equation*}
and hence we obtain the desired result.
\end{proof}

Also the sink-component map, Theorem~\ref{thm:7.4} and Theorem~\ref{thm:AC_LG} give an analogue of \cite[Corollary 7.4]{GZ83} for the chromatic polynomial.
\begin{thm} \label{thm:AC_SD}
The acyclic orientation polynomial of $G=(V,E)$ is given by
\begin{equation} \label{eq:AC_SD}
A_G(V)=\sum_{\mathfrak{o} \in \mathcal{A}(G)}\prod_{ B \in \pi(\mathfrak{o})}v_{B},
\end{equation}
where $\pi$ is the sink-component map.
\end{thm}
\begin{proof}
A proof of the theorem follows from
\begin{align*}
A_G(V) &= \sum_{\pi\in L_G} (-1)^{d-|\pi|} \mu_G(\hat{0}, \pi) \prod_{B\in\pi} v_B && \mbox{by Theorem~\ref{thm:AC_LG},} \\
&= \sum_{\pi\in L_G} \sum_{\substack{\oo\in\AO(G) \\ \pi(\oo)=\pi}} \prod_{B\in\pi} v_B && \mbox{by Theorem~\ref{thm:7.4},} \\
&= \sum_{\oo\in\AO(G)}\prod_{B\in\pi(\oo)} v_B.
\end{align*}
\end{proof}

\section{Chromatic symmetric functions and Stanley's sink theorem}\label{sec:Stanley}
In this section, we present a \emph{new} proof for \cite[Theorem 3.3]{Sta95}, which expresses the number of acyclic orientations with a fixed number of sinks as the sum of the coefficients of elementary symmetric functions $e_{\lambda}$ with a fixed length in the expansion of the chromatic symmetric function. Our proof does not require the theory of quasi-symmetric functions and $P$-partitions, which was used in the original proof of \cite[Theorem 3.3]{Sta95}.

We begin with the definition of the \emph{chromatic symmetric function} $X_G$ of a graph $G=(V,E)$. Let $x_1,x_2, \dots$ be commuting indeterminates. A \emph{proper
coloring} $\kappa$ of $G$ is a function $\kappa:V \to \{1,2,3, \dots \}$ such that $\kappa(v) \ne \kappa(v') $ whenever $v,v'$ are adjacent.
\begin{defn}[{\cite[Definition 2.1]{Sta95}}] The \emph{chromatic symmetric function} $X_G$
is defined as
\[
X_G=\sum_{\kappa}{\prod_{v \in V}x_{\kappa(v)}},
\] 
where the sum is over all proper colorings $\kappa$ of $G$.
\end{defn}

We will expand the symmetric function $X_G$ in terms of power sum symmetric functions and elementary symmetric functions. For $n \geq 1$, the $n$-th \emph{power sum symmetric function} $p_n$ and \emph{elementary symmetric function} $e_n$ are 
\[
p_n=\sum_{i \geq 1}{x_i^n} \quad\mbox{ and }\quad e_n=\sum_{i_1<\cdots<i_n}x_{i_1} \cdots x_{i_n},
\]
with $p_0=e_0=1$.
A sequence $\alpha= (\alpha_1 ,\alpha_2,  \dots , \alpha_j)$ of positive integers with $\alpha_1+\dots+\alpha_j=n$ is a \emph{composition} of $n$, and the set of compositions of $n$ is denoted by $\operatorname{Comp}(n)$. 
The \emph{length} of $\alpha$ is defined by $\ell(\alpha)=j$. 
A composition $\lambda=(\lambda_1,\dots,\lambda_j)$ is called a \emph{partition} if $\lambda_1\ge \lambda_2 \ge \cdots \ge \lambda_j$, and we write $\lambda\vdash n$.
For a composition $\alpha=(\alpha_1,\dots,\alpha_j)$, define 
\[
p_\alpha=\prod_{i=1}^{j}p_{\alpha_i}=p_{\alpha_1}\cdots p_{\alpha_j}, \quad\mbox{ and }\quad e_\alpha=\prod_{i=1}^{j}e_{\alpha_i}=e_{\alpha_1}\cdots e_{\alpha_j}.
\]
Note that $\{p_{\lambda} \mid \lambda\vdash n \}$ and $ \{e_{\lambda} \mid \lambda\vdash n \}$ form bases for the space of homogeneous symmetric functions of degree $n$.

Let us collect expansions of $X_G$ with respect to the power sum symmetric functions $p_\lambda$:
\begin{align}
X_G &= \sum_{S \subseteq E}(-1)^{|S|}p_{\lambda(S)} &&  \mbox{{\cite[Theorem 2.5]{Sta95}}} \label{thm:2.5} \\  
&= \sum_{S \in \BB_G}(-1)^{|S|}p_{\lambda(S)} &&  \mbox{{\cite[Theorem 2.9]{Sta95}}} \label{thm:2.9} \\ 
&= \sum_{\pi \in L_G}\mu(\hat{0},\pi)p_{{\operatorname{type}(\pi)}} && \mbox{{\cite[Theorem 2.6]{Sta95}}} \label{thm:2.6} \\ 
&= \sum_{\mathfrak{o} \in \mathcal{A}(G)}(-1)^{d-|\pi(\mathfrak{o})|}p_{{\operatorname{type}(\pi(\mathfrak{o}))}} && \mbox{{\cite[Proposition 5.2]{BN20}}} \label{prop:5.2},  
\end{align}
where $\lambda(S)$ and $\operatorname{type}(\pi)$ are the non-increasing sequences of the sizes of connected components of a spanning subgraph $S$, and elements of a vertex partition $\pi$, respectively. 

The following proposition shows that in the expansion of $p_n$ with respect to $e_{\lambda}$, the sum of coefficients corresponding to partitions of a fixed length is a binomial coefficient up to sign. The proof below uses the generating functions. For its determinantal expression, see e.g. \cite[Exercise I.2.8]{Mac98}.
\begin{prop}\label{prop:power_ele}  Let 
$
p_n=\sum_{\lambda\vdash n}b_{\lambda}e_{\lambda}
$ 
be the expansion of $p_n$ in terms of $e_{\lambda}$. Then 
\[
\sum_{ \substack{\lambda\vdash n \\ \ell(\lambda)=j}}b_{\lambda}=(-1)^{n-j}{{n} \choose j}.
\]
\end{prop}
\begin{proof}
The generating function for $(-1)^{n-1}p_n$ is
\begin{align*}
\sum_{n\ge 1} p_n(-z)^{n-1} &= \sum_{i\ge 1} \frac{x_i}{1+x_i z} = \frac{\mathrm{d}}{\mathrm{d}z} \log \prod_{i\ge 1} (1+x_i z) \\
&= \frac{\mathrm{d}}{\mathrm{d}z}\log E(z) = \frac{E'(z)}{E(z)} \\
&= \left( \sum_{n\ge 1} n e_n z^{n-1} \right) \sum_{j\ge 1} \left(- \sum_{n\ge 1} e_n z^n \right)^{j-1},
\end{align*}
where $E(z)=\sum_{n \geq 0}e_n z^n$. Equating the coefficients of $z^{n-1}$ on both sides yields
\[
p_n=\sum_{\alpha \in \operatorname{Comp}(n)}{(-1)^{n-\ell(\alpha)}\alpha_1e_{\alpha}}. 
\]
The proposition then follows from the following identities:
\begin{align*}
\sum_{ \substack{ \alpha \in \operatorname{Comp}(n) \\ \ell(\alpha)=j}}\alpha_1
&= \sum_{ \substack{ \alpha \in \operatorname{Comp}(n) \\ \ell(\alpha)=j}} |\{(a,\alpha_1+1-a,\alpha_2,\alpha_3, \dots) \mid 1 \leq a \leq \alpha_1\}| \\
&= \sum_{ \substack{ \alpha \in \operatorname{Comp}(n+1) \\ \ell(\alpha)=j+1}} 1
= \binom{n}{j}.
\end{align*}
\end{proof}
\begin{ex} For the graph $H$ in Example~\ref{ex:SGF}, the expansion of $X_H$ with respect to $p_{\lambda}$ is given as follows (we write for example $211$ as short for $(2,1,1)$):
\[
X_H=-4p_{4}+6p_{31}+2p_{22}-5p_{211}+p_{1111}.
\]
From the relations established in the proof of Proposition~\ref{prop:power_ele},
\[
p_1=e_1,\, p_2=e_{11}-2e_2,\, p_3=e_{111}-3e_{21}+3e_3,\,p_4=e_{1111}-4e_{211}+4e_{31}+2e_{22}-4e_4,
\]
we derive $X_H=16e_4+2e_{31}$.
\end{ex}

Let $\Lambda$ be the $\mathbb{Q}$-algebra of symmetric functions with $\mathbb{Q}$-coefficients. Since $\{e_n\}_{n \geq 1}$ is algebraically independent and generates $\Lambda$ as a $\mathbb{Q}$-algebra, assigning a value to each $e_n$ determines an algebra homomorphism on $\Lambda$. Let $\mathbb{Q}[t]$ be the ring of polynomials in an
indeterminate $t$ with $\mathbb{Q}$-coefficients. Define the algebra homomorphism $\phi: \Lambda \to \mathbb{Q}[t]$ by $\phi(e_n)=t$ for each $n \geq 1$. Then, for each partition $\lambda$,
\begin{equation}\label{eqn:length}
\phi(e_{\lambda}) = t^{\ell(\lambda)}.
\end{equation}
For this homomorphism $\phi$, the image $\phi(p_n)$ is computed as follows.
\begin{lem}[{\cite[Exercise 7.43]{Sta99}}]\label{lem:image}
The image of $p_n$ under $\phi$ is equal to
\begin{equation*}
\phi(p_n) = (-1)^{n-1}(1-(1-t)^n).
\end{equation*}
\end{lem}
\begin{proof}
By Proposition~\ref{prop:power_ele}, for $p_n=\sum_{\lambda\vdash n}{b_\lambda e_\lambda}$,
\[
\phi(p_n) =\sum_{j=1}^{n}\Big(\sum_{\substack{\lambda\vdash n\\ \ell(\lambda)=j}}b_{\lambda}\phi(e_\lambda)\Big)
=\sum_{j=1}^{n}(-1)^{n-j}{n \choose j} t^j =(-1)^{n-1}(1-(1-t)^n).
\] 
\end{proof}

We are ready to prove Stanley's sink theorem {\cite[Theorem 3.3]{Sta95}}. The proof below uses Theorem~\ref{thm:AC_Sub} and equation \eqref{thm:2.5} involving subgraph expansions.
Instead, one can employ Theorem~\ref{thm:AC_BC} and equation~\eqref{thm:2.9} concerning broken circuit complexes, Theorem~\ref{thm:AC_LG} and equation~\eqref{thm:2.6} concerning bond lattices, or Theorem~\ref{thm:AC_SD} and equation~\eqref{prop:5.2} concerning concerning the sink-component map defined in Section~\ref{sec:AC_SD}.

\begin{thm}[{\cite[Theorem 3.3]{Sta95}}]\label{thm:Stanley} Let $X_G=\sum_{\lambda\vdash d} c_{\lambda}e_{\lambda}$
be the expansion of the chromatic symmetric function $X_G$ in terms of the elementary symmetric functions $e_{\lambda}$, and let $\operatorname{sink}(G,j)$ be the number of acyclic orientations of $G$ with $j$ sinks. Then 
\begin{equation}\label{eq:Stanley}
\operatorname{sink}(G,j)=\sum_{\substack{\lambda\vdash d \\ \ell(\lambda)=j}} c_\lambda.
\end{equation}
\end{thm}
\begin{proof}
The statement that equation~\eqref{eq:Stanley} holds for every $j \geq 1$ is equivalent to
\begin{equation}\label{eq:A_G_ele}
a_G(t)=\sum_{\lambda\vdash d}c_{\lambda}t^{\ell(\lambda)}. 
\end{equation}
To prove equation~\eqref{eq:A_G_ele}, we apply $\phi$ to the two expansions of $X_G$ in terms of $e_\lambda$ and $p_\lambda$. By equation \eqref{eqn:length}, 
\[
\phi(X_G)
= 
\phi\left(\sum_{\lambda\vdash d} c_{\lambda}e_{\lambda} \right)
=
\sum_{\lambda\vdash d} c_{\lambda}\phi(e_{\lambda})
=
\sum_{\lambda\vdash d} c_{\lambda}t^{\ell(\lambda)}.
\] 
On the other hand, we obtain
\begin{align*}
\phi(X_G) &= \phi\left( \sum_{S\subseteq E} (-1)^{|S|} p_{\lambda(S)} \right) \\
&= \sum_{S\subseteq E} (-1)^{s(S)} \prod_{C\in\CC(S)} (-1)^{|V(C)|-1}\phi(p_{|V(C)|}) \\
&= \sum_{S\subseteq E} (-1)^{s(S)} \prod_{C\in\CC(S)} \left( 1-(1-t)^{|V(C)|} \right) \\
&= a_G(t),
\end{align*}
where the first equality uses equation \eqref{thm:2.5}, the second one follows from the definition of the corank $s(S)=|S|-d+|\mathcal{C}(S)|$, the third one is verified by Lemma~\ref{lem:image}, and the last one follows from equation \eqref{eq2:Sub}.
This completes the proof.
\end{proof}

\begin{ex}
Let $H$ be the graph in Example~\ref{ex:SGF}. Since $X_H=16e_4+2e_{31}$, the previous theorem gives $\operatorname{sink}(H,1)=16$ and $\operatorname{sink}(H,2)=2$, or equivalently $a_H(t)=16t+2t^2$. This coincides with the computation in Example~\ref{ex:SGF}.
\end{ex}

%%%%%%%%%%%%%%%%%%%%%%%%%%%%%%%%%%%%%%%%%%%%%%%%%%%%%%%%%
%%%%%%%%%%%%%%%%%%%%%%%%%%%%%%%%%%%%%%%%%%%%%%%%%%%%%%%%%
\section{On distinguishing graphs by acyclic orientation polynomials} \label{sec:invariant}

We end this paper with a discussion on how well graphs are distinguished by the polynomials $A_G(V)$ and $a_G(t)$ in comparison to $X_G$. In this section, we assume that $G$ is a simple graph.
Stanley \cite{Sta95} asked the question of whether the chromatic symmetric polynomial $X_G$ distinguishes non-isomorphic trees.
The answer to the question is still unknown, but several graph polynomials \cite{GS01, MMW08} associated with $X_G$ have been introduced in the endeavor to resolve the question. For example, the noncommutative version $Y_G$ of $X_G$ (defined in \cite{GS01}) is a complete isomorphism invariant, i.e., $Y_G$ distinguishes among all simple graphs $G$ (\cite[Proposition 8.2]{GS01}).

The (univariate) polynomial $a_G(t)$ is a \emph{weaker} invariant than $X_G$ by Stanley's sink theorem {\cite[Theorem 3.3]{Sta95}}. Trees with $10$ or fewer vertices are distinguished by $a_G(t)$. But there exist non-isomorphic trees having the identical $a_G(t)$ with $11$ vertices. These two trees shown in Figure~\ref{fig:instances} were introduced in \cite{EG06} as the smallest instances which are not distinguished by the subtree data. 

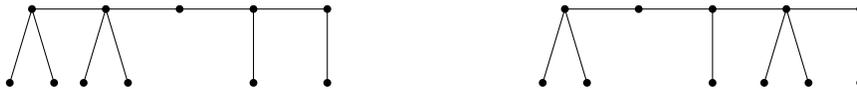
\begin{figure}[h]
\centering
\resizebox{.9\linewidth}{!}{
\setlength{\tabcolsep}{40pt}
    \begin{tabular}{cc}
         \begin{tikzpicture}
         %-------NODES----------------
         \fill (1,1) circle (1.5pt);
         \fill (2,1) circle (1.5pt);
         \fill (3,1) circle (1.5pt);
         \fill (4,1) circle (1.5pt);
         \fill (5,1) circle (1.5pt);
         \fill (1-0.3,0) circle (1.5pt);
         \fill (1+0.3,0) circle (1.5pt);
         \fill (2-0.3,0) circle (1.5pt);
         \fill (2+0.3,0) circle (1.5pt);
         \fill (4,0) circle (1.5pt);
         \fill (5,0) circle (1.5pt);
         %-------EDGES----------------
         \draw (1,1)--(2,1);
         \draw (2,1)--(3,1);
         \draw (3,1)--(4,1);
         \draw (4,1)--(5,1);
         
         \draw (1,1)--(1-0.3,0);
         \draw (1,1)--(1+0.3,0);
         \draw (2,1)--(2-0.3,0);
         \draw (2,1)--(2+0.3,0);
         \draw (4,1)--(4,0);
         \draw (5,1)--(5,0);
         
%%         \node at (3,-0.5) {$T_1$};
         \end{tikzpicture}&
         \begin{tikzpicture}
         %-------NODES----------------
         \fill (1,1) circle (1.5pt);
         \fill (2,1) circle (1.5pt);
         \fill (3,1) circle (1.5pt);
         \fill (4,1) circle (1.5pt);
         \fill (5,1) circle (1.5pt);
         \fill (1-0.3,0) circle (1.5pt);
         \fill (1+0.3,0) circle (1.5pt);
         \fill (3,0) circle (1.5pt);
         \fill (4-0.3,0) circle (1.5pt);
         \fill (4+0.3,0) circle (1.5pt);
         \fill (5,0) circle (1.5pt);
         %-------EDGES----------------
         \draw (1,1)--(2,1);
         \draw (2,1)--(3,1);
         \draw (3,1)--(4,1);
         \draw (4,1)--(5,1);
         
         \draw (1,1)--(1-0.3,0);
         \draw (1,1)--(1+0.3,0);
         \draw (3,1)--(3,0);
         \draw (4,1)--(4-0.3,0);
         \draw (4,1)--(4+0.3,0);
         \draw (5,1)--(5,0);

%%         \node at (3,-0.5) {$T_2$};
         \end{tikzpicture}
    \end{tabular}
}
    \caption{Two non-isomorphic trees with the same $a_G(t)$.}  \label{fig:instances}  
\end{figure}

The (multivariate) acyclic orientation polynomial $A_G(V)$ is a \emph{complete} isomorphism invariant.
To see this, we first assume that $G$ is connected. For two vertices $u,v \in V$, the coefficient of $uv$ in $A_G(V)$ is zero if and only if $u$ and $v$ are adjacent. Hence, the polynomial $A_G(V)$ can recover the original graph $G$.
For any graph $G$, the polynomial $A_G(V)$ has a factorization into those of its connected components, and therefore using the previous argument applied to each component shows that $G$ can be determined by $A_G(V)$.

\section*{Acknowledgment}
We would like to thank Jang Soo Kim, Bruce Sagan and Thomas Zaslavsky for helpful comments and corrections.
We are also grateful to the anonymous referees for many helpful comments regarding correcting errors, and improving expositions. In particular, one referee pointed out the connection (Appendix~\ref{sec:appendix}) of 
the acyclic orientation polynomial with the $\VV$-polynomial.

The fourth author was supported by Basic Science Research Program through the National Research Foundation of Korea(NRF) funded by the Ministry of Education(2020R1A6A3A13076804).

\appendix
\section{The $\VV$-polynomial} \label{sec:appendix}
In this appendix, we present a relation between the acyclic orientation polynomial and the $\VV$-polynomial.

In \cite{EMM11}, to study a relationship between the Tutte polynomial and the Potts model, Ellis-Monaghan and Moffatt introduced the \emph{$\VV$-polynomial} $\VV(G)$ of a graph $G$.
This polynomial is defined recursively by a ``general'' deletion-contraction recurrence. We will show that the acyclic orientation polynomial $A_G(V)$ is a specialization of that polynomial.

Let $G=(V,E)$ be a graph with multiple edges or loops allowed. We may assume that $V=\{v_1,v_2,\dots, v_d\}$. To define the $\VV$-polynomial of G, we need the following data:
\begin{enumerate}
\item[•] $(S,\circ)$ is a torsion-free commutative semigroup.
\item[•] $\omega$ is a $S$-valued vertex weight function, that is, $\omega: V\rightarrow S$.
\item[•] $\xx=\{x_k\}_{k\in S}$ and $\bgamma=\{\gamma_e\}_{e\in E}$ are formal commuting variables.
\end{enumerate}
For a graph $G$ equipped with an $S$-valued vertex weight $\omega$ and $e\in E$, let $G\setminus e$ be the graph obtained from $G$ by removing $e$, leaving the vertex weight unchanged. Then $E(G\setminus e) = E\setminus \{e\}$. When $e$ is not a loop, let $G/e$ be the graph obtained from $G$ by contracting $e$ and changing the vertex weight $\omega$ as follows: if $v_i$ and $v_j$ are the vertices incident to $e$, and $v_e$ is the vertex of $G/e$ created by the contraction, then $\omega(v_e) := \omega(v_i)\circ\omega(v_j)$. Loops are not contracted, and then again $E(G/ e) = E\setminus \{e\}$.
\begin{defn}
The \emph{$\VV$-polynomial} $\VV(G) = \VV(G,\omega; \xx,\bgamma)$ of $G$ is a polynomial in $\ZZ[\xx, \bgamma]$ defined recursively by:
\[
\VV(G) = 
\begin{dcases}
\prod_{i=1}^d x_{\omega(v_i)}, & \mbox{if $G$ consists of $d$ isolated vertices,} \\
\VV(G\setminus e) + \gamma_e \VV(G/e), & \mbox{if $e$ is not a loop,} \\
(\gamma_e+1) \VV(G\setminus e), & \mbox{if $e$ is a loop.}
\end{dcases}
\]
\end{defn}
Due to Ellis-Monaghan and Moffatt, the $\VV$-polynomial is well-defined~\cite[Proposition 3.2]{EMM11}, i.e., it is independent of the order in which the deletion-contraction relation is applied to the edges. They also gave a subgraph expansion of $\VV(G)$~\cite[Theorem 3.3]{EMM11}:
\begin{equation} \label{eq:V_subgraph}
\VV(G) = \sum_{S\subseteq E} \prod_{C\in\CC(S)} x_{\omega(C)} \prod_{e\in S} \gamma_e,
\end{equation}
where $\omega(C) = \omega(v_{i_1})\circ\dots\circ\omega(v_{i_k})$ if $C$ consists of vertices $\{v_{i_1},\dots,v_{i_k}\}$.

Now set $S=\ZZ[V]$, the ring of polynomials in $V$, and define an operator $\circ$ on $S$ by $a\circ b = a + b - ab$ for $a,b\in S$, where addition and multiplication are those in $\ZZ[V]$. Then one can easily check that $(S,\circ)$ forms a torsion-free commutative semigroup. We also define a vertex weight $\omega$ given by $\omega(v) = v\in \ZZ[V]$ for each $v\in V$. Then $\VV(G)$ can be defined with these data.
When we specialize $\VV(G)$ at $x_k=-k$ for each $k\in S$ and $\gamma_e= -1$ for each $e\in E$, the polynomial $\VV(G)$ belongs to $\ZZ[V]$, and satisfies the following recurrence:
\[
\VV(G)\big{|}_{\substack{x_k=-k \\ \gamma_e=-1}} = 
\begin{dcases}
(-1)^d \prod_{i=1}^d v_i, & \mbox{if $G$ consists of $d$ isolated vertices,} \\
\VV(G\setminus e)\big{|}_{\substack{x_k=-k \\ \gamma_e=-1}} - \VV(G/e)\big{|}_{\substack{x_k=-k \\ \gamma_e=-1}}, & \mbox{if $e$ is not a loop,} \\
0, & \mbox{if $G$ has a loop.}
\end{dcases}
\]
Since the operator $\circ$ on $S$ coincides with the relation \eqref{eq:edge_relation}, and the polynomials $A_G(V)$ and $\VV(G)$ with the above specialization satisfy the same recurrence relation (the relation \eqref{eq:DC}) (up to sign) and the initial condition (up to sign), we deduce that
\[
A_G(V) = (-1)^d \VV(G)\big{|}_{\substack{x_k=-k \\ \gamma_e=-1}}.
\]
Also with the same specialization, we could directly obtain the subgraph expansion (Theorem~\ref{thm:AC_Sub}) of $A_G(V)$ from equation \eqref{eq:V_subgraph}.

\bibliographystyle{alpha}
\bibliography{AOS}

\end{document}